\documentclass[a4paper]{article}

\usepackage[pages=all, color=black, position={current page.south}, placement=bottom, scale=1, opacity=1, vshift=5mm]{}

\usepackage{natbib}
\usepackage[margin=1in]{geometry} 

\usepackage{amsmath}
\usepackage{amsthm}
\usepackage{amssymb}
\usepackage{bm}

\usepackage[utf8]{inputenc}
\usepackage{hyperref}
\hypersetup{
	unicode,
	pdfauthor={Author One, Author Two},
	pdftitle={A Vector Space Approach to HTA},
	pdfsubject={Extreme Value Theory},
	pdfkeywords={Extremes, Regular Variation},
	pdfproducer={LaTeX},
	pdfcreator={pdflatex}
}

  \theoremstyle{plain}
\newtheorem{theorem}{Theorem}
\newtheorem{corollary}{Corollary}
\newtheorem{lemma}{Lemma}
\newtheorem{proposition}{Proposition}
  \theoremstyle{remark}
\newtheorem{remark}{Remark}

  \theoremstyle{definition}
\newtheorem{example}{Example}
\newtheorem{definition}{Definition}

\usepackage{graphicx, xcolor}
\graphicspath{{fig/}}

\DeclareMathOperator{\sign}{sign}

\usepackage{ulem}  

\usepackage{algorithm, algpseudocode} 

\usepackage{mathrsfs} 

\title{A Vector Space Approach to Heavy Tailed Analysis}
\author{Kenneth Broadhead$^1$ \and Daniel Cooley$^2$}

\date{
	Colorado State University \\ \texttt{kenneth.broadhead@colostate.edu$^1$, 
        cooleyd@colostate.edu$^2$}\\%
}

\begin{document}
	\maketitle

    \textcolor{red}{\textbf{DISCLAIMER: We have become aware of technical issues surrounding our argument for closure under addition in our vector space. We are working on understanding and correcting this issue. - The Authors}} 
    
	\begin{abstract}

    We construct a vector space whose defining characteristics are rooted in univariate regular variation of random variables. Specifically, the base vector space $\mathbb{V}_b$ consists of random variables whose limiting tail probabilities, when scaled by regularly varying functions of the form $b(s)=s^\alpha L(s)$, are finite. Defining a subspace ${\cal N}_b$ corresponding to random variables in $\mathbb{V}_b$ whose limiting tail probabilities are zero when normalized by $b(s)$ allows the base space $\mathbb{V}_b$ to be partitioned into equivalence classes. We define a vector space $\mathbb{W}_b$ consisting of these equivalence classes, and show its nonzero elements are equivalence classes of regularly varying random variables. We show that a natural norm exists for $\mathbb{W}_b$ if $\alpha > 1$. We show that the equivalence classes and convergence in norm are different than more familiar vector spaces of random variables. Turning our attention to extreme value modeling, we consider finite-dimensional subspaces of $\mathbb{W}_b$ whose basis vectors are jointly regularly varying. We show that in the case $\alpha = 2$, the previously defined tail pairwise dependence measure serves as an inner product. As any finite-dimensional space is complete, we can use the projection theorem to perform linear prediction.

	
    \end{abstract}

	\tableofcontents

	\section{Introduction}\label{sec:1}
	\label{sec:intro}

	\subsection{Background and Motivation} \label{sec:1.1}

	Vector space methods see frequent use in classical areas of statistics. 
    Linear models, many time series methods, and covariance-based multivariate analysis can all be developed in terms of vector spaces of random variables. Additional structure on the vector space, in particular an inner product, provides utility to develop statistical methods. In this paper, we develop a vector space arising from univariate regular variation of random variables, with the aim of constructing new statistical methods for extreme value analysis (EVA). As regular variation is an asymptotic property giving structure only to the tail of a random variable, a vector space developed from regular variation will be different than familiar examples. After a space is constructed, we develop some of its properties and further connect this space to the tail pairwise dependence measure.
	
	In the remainder of section 1, we review univariate regular variation, multivariate regular variation, and the tail pairwise dependence measure.
    In section \ref{sec:2}, we construct a vector space of tail equivalent regularly varying random variables.
    In section \ref{sec:3}, we show that if $\alpha > 1$ there exists a natural norm on the vector space constructed in section 2. 
    In section \ref{sec:4}, we restrict attention to model-oriented finite dimensional subspaces of the vector space in section 2. We show that these subspaces are the appropriate setting for the tail pairwise dependence measure, which becomes an inner product on these spaces. This provides a modeling framework for extreme value analysis analogous to classical covariance modeling. In section \ref{sec:5}, we discuss some possible applications, extensions, and avenues for future research.

    \subsection{Univariate Regular Variation and $\alpha$-Scale} 

    Regular variation of random variables is a familiar framework for working with heavy tailed random variables. In essence, if a random variable is regularly varying, its distribution's tails behave asymptotically like power laws. Below, we outline an approach to univariate regular variation that will be used throughout. We begin with some familiar definitions. 

    A nonnegative measurable function $L:(0,\infty)\rightarrow (0,\infty)$ is said to be slowly varying (at infinity) if for all $x>0$
    \[
    \lim_{s\rightarrow \infty} \frac{L(sx)}{L(s)} = 1.
    \]
    A function $f$ of the form $f(s)=s^{\alpha} L(s)$ for $\alpha \in \mathbb{R}$ is said to be regularly varying with index $\alpha$. 
    
    \begin{definition}[Regular Variation] \label{def:regvar}
    	A real-valued random variable $X$ is regularly varying with tail index $\alpha>0$ if there exists a regularly varying function $b(s)=s^\alpha L_b(s)$, and nonnegative real numbers $\kappa_+, \kappa_{-}$, with $\kappa_++\kappa_{-} > 0$ such that for all $x>0$
    	\[
    	\lim_{s\rightarrow \infty} b(s) P(X>sx) = \kappa_+ x^{-\alpha}, \label{eq:1.2.1} \tag{1.2.1}
    	\]
    	\[
    	\lim_{s\rightarrow \infty} b(s) P(X<-sx) = \kappa_{-} x^{-\alpha}.  \label{eq:1.2.2} \tag{1.2.2} 
    	\]
    \end{definition}
    \begin{remark}
        One does not need to require the function $b$ be regularly varying, for any function positive measurable function $b(s)\rightarrow \infty$ for which \ref{eq:1.2.1} and \ref{eq:1.2.2} holds can be shown to be regularly varying.
    \end{remark}
    \begin{remark}
    	Classical treatments of regular variation often focus on positive random variables, and use the function $b(s) = P(X > s)^{-1}$ so that $\kappa=1$. For real valued random variables, $b(s)=P(|X|>s)^{-1}$ is similarly used, and additional tail balance conditions are added to exclude pathological behavior. We accomplish the same by treating each tail separately. Compare with the presentation of regular variation found in \cite{EKM97} and \cite{KulikSoulier2020}.
    \end{remark}
    \begin{remark}
    	Note, if $X$ is regularly varying using function the function $b$, so is $|X|$. If both \eqref{eq:1.2.1} and \eqref{eq:1.2.2} are satisfied, then we have for all $x>0$
    	\[
    	\lim_{s\rightarrow \infty} b(s) P(|X|>sx) = \lim_{s\rightarrow \infty} b(s) (P(X>sx) + P(X<-sx)) = (\kappa_+ + \kappa_{-})x^{-\alpha}.
    	\]
    \end{remark}

    In what is to follow, regularly varying functions with positive index will play an important role, so we give them a name.

    \begin{definition}
        A normalizing function is a (regularly varying) function of the form $b(s)=s^\alpha L(s)$, for some $\alpha>0$ and slowly varying function $L$.
    \end{definition}

    While the normalizing function $b$ in definition \ref{def:regvar} is not unique, there is great utility in fixing a canonical normalizing function $b$ and considering collections of random variables that are regularly varying using $b$ as a normalizing function. Under a single fixed normalizing function, the constants $\kappa, \kappa_+$ and $\kappa_{-}$ are determined. This will be the approach used throughout the remainder of the paper. In particular, once fixed, the normalizing function $b$ will play a central role in defining the vector space in section \ref{sec:2}. Similarly, the constants $\kappa, \kappa_+, \kappa_{-}$ corresponding to a regularly varying random variable with fixed normalizing function will play an important role. In particular, they will be instrumental in defining a natural norm for our vector space in section \ref{sec:3}. As these values will be consequential, we give them a name. 

    \begin{definition}
	   Let $X$ be regularly varying with normalizing function $b$. We call $\kappa := \kappa_+ + \kappa_{-}$ the $\alpha$-scale of $X$ when normalized by $b$. Similarly, we refer to $\kappa_+$ and $\kappa_{-}$, respectively, as the left and right tail $\alpha$-scale of $X$ when normalized by $b$. When the normalizing function $b$ is clear from context, we may simply refer to $\kappa$ as the $\alpha$-scale of $X$, and similarly for $\kappa_+$ and $\kappa_{-}$.
    \end{definition}
    
    \subsection{Multivariate Regular Variation and Dependence Measures} \label{sec:1.4}

    While not central to the construction of our vector space in Sections \ref{sec:2} and \ref{sec:3}, multivariate regular variation does link our work to familiar areas of extreme value analysis (see Section \ref{sec:4}). Here we briefly outline a framework for multivariate regular variation. See the classic references \cite{Res86, Res07} for details. 

    A random vector $\bm X$ taking values in $\mathbb{R}^d$ is regularly varying if there exists a normalizing function $b(s)=s^{\alpha} L(s)$ such that
    \begin{equation}\label{eq:1.3.1}
    	b(s) P(s^{-1} \bm X \in \cdot) \stackrel{v}{\rightarrow} \nu_{\bm X}(\cdot). \tag{1.3.1}
    \end{equation}
    Here ``$\stackrel{v}{\rightarrow}$" denotes vague convergence in $M_+([-\bm{\infty}, \bm{\infty}]^{d} \setminus \{\bm 0\})$, the space of nonzero Radon measures defined on the Borel sets of $[-\bm{\infty}, \bm{\infty}]^{d} \setminus \{\bm 0\}$.
    The limit measure $\nu_{\bm{X}}$ is commonly called the exponent measure and does not put mass on $[-\bm{\infty}, \bm{\infty}]^{d} \setminus (-\bm{\infty}, \bm{\infty})^d$. 
    A comprehensive development of multivariate regular variation begins by defining vague convergence of measures which leads to a sequential definition of regular variation, to which (\ref{eq:1.3.1}) is known to be equivalent \cite{Res86}.
    Although \ref{eq:1.3.1} will be sufficient for our purposes, we briefly mention that more general approaches to regular variation exist, and can extend regular variation to spaces beyond $\mathbb{R}^d$, see \cite{HultLindskog2006, BP19, KulikSoulier2020}.

	The exponent measure is $-\alpha$-homogeneous, and admits the following decomposition when $\bm{X}$ is expressed in polar coordinates:
	\[
	\nu_{\bm{X}}(\{x\in\mathbb{R}^d: ||x||>r, \frac{x}{||x||}\in B\}) = r^{-\alpha} H_{\bm{X}}(B).
	\]
	Here, $H_{\bm{X}}$ is called the angular measure, and is defined on the unit sphere $\mathbb{S}^{d-1}$ corresponding to the chosen norm with its standard Borel $\sigma$-algebra. 
	As we allow an abstract normalizing function $b$ in formulating regular variation above, the angular measure $H_{\bm{X}}$ may not be a probability measure. The angular measure provides a full characterization of extremal dependence between the components of the random vector $\bm{X}$.
	
	In later sections, we will be focusing on summary measures of dependence for bivariate regularly varying random variables. For $X_i$ and $X_j$ univariate components of $\bm X$, we consider
	\[
	\sigma(X_i, X_j) = \int_{\mathbb{S}^1} \mathrm{sgn(\omega_i)}  \mathrm{sgn(\omega_j)} |\omega_i|^{\frac{\alpha}{2}} |\omega_j|^{\frac{\alpha}{2}} dH_{\bm{X}}(\bm{\omega}). \label{eq:1.3.2} \tag{1.3.2}
	\]
	A similar measure was first defined by \cite{LR12} under the name extremal dependence measure (EDM). They considered a nonnegative bivariate regularly varying random vector, their definition specified $H$ to be a probability measure, and the EDM did not depend on the tail index $\alpha$. 
	Compare with \cite{SR04}, section 3, for a closely related dependence measure by the same name. 
	Still in the nonnegative setting, \cite{CT19} dropped the requirement $H$ be a probability measure, and required that the tail index be $\alpha=2$, simplifying \ref{eq:1.3.2}. 
	They then collected pairwise extremal dependence measures of a general $d$ dimensional regularly varying vector into a matrix, defining the tail pairwise dependence matrix (TPDM). 
	\cite{KZ24} extended the tail pairwise dependence matrix to accommodate general $\alpha$ by allowing the underlying dependence measure to depend on $\alpha$, obtaining a measure similar to \ref{eq:1.3.2}, and investigated the properties of this extension. 
	Definition \ref{eq:1.3.2} effectively extends the definition of \cite{KZ24} beyond the nonnegative setting, as was done in \cite{mhatre2024}.
	
	Multivariate regular variation will not be used until section \ref{sec:4}, in which we will return to the dependence measure $\sigma(\cdot, \cdot)$ defined by \ref{eq:1.3.2}. As our notation suggests, we will view $\sigma(\cdot,\cdot)$ as a bivariate function defined on a suitable space, and we will adopt the term tail pairwise dependence measure to refer to $\sigma(\cdot, \cdot)$ as defined in \ref{eq:1.3.2}.

	\section{A Vector Space for Regular Variation} \label{sec:2}
	
	\subsection{A General Construction} \label{sec:2.1}

    The familiar $L^p$ spaces consist of equivalence classes of random variables on a probability space with finite $p^{th}$ moment. Random variables which differ only on a set of measure zero are identified; that is, mapped to the same equivalence class. Our vector space, developed using ideas from univariate regular variation theory, will similarly consist of equivalence classes, but equivalence will be defined by asymptotic tail behavior. Consequently, our equivalence classes will be much more coarse than those in $L^p$. Before we consider equivalence classes, we first establish a suitable space on which to define them.
	
	\begin{definition}
		For a fixed probability space $(\Omega, \Sigma, \mathbb{P})$, denote the collection of all random variables on $\Omega$ by $L^0(\Omega, \Sigma, \mathbb{P})$. Let $b(s)$ be a (regularly varying) normalizing function with index $\alpha>0$, and define	
		\[
		\mathbb{V}_b = \{X\in L^0(\Omega, \Sigma, \mathbb{P}) : \forall x>0, \lim_{s\rightarrow \infty} b(s)P(X>sx)<\infty \text{ and}  \lim_{s\rightarrow \infty} b(s)P(X<-sx)<\infty\}.
		\]
	\end{definition}

    Importantly, the space is defined by the particular normalizing function, and we include the subscript $b$ to make this dependence explicit.

	\begin{proposition}
	The set $\mathbb{V}_b$ is a real vector space under the usual operations of addition and scalar multiplication of random variables. 
	\end{proposition}
	
	\begin{proof}		
		First we show closure under addition. Let $X,Y \in \mathbb{V}_b$. We then have
		\[
		\lim_{s\rightarrow \infty} b(s)P\bigg(X+Y>sx\bigg) \leq \lim_{s\rightarrow \infty}
		b(s)\bigg[P\bigg(X>\frac{sx}{2}\bigg)+P\bigg(Y>\frac{sx}{2}\bigg) \bigg]
		< \infty.
		\]
		The first inequality follows from the fact that $\{X+Y>sx\}\subset \{X>\frac{sx}{2}\}\cup\{Y>\frac{sx}{2}\}$; while the second follows from the fact that each limit is less than infinity for each $x' = \frac{x}{2}>0$.
		
		A similar argument holds for the left tail:
		\[
		\lim_{s\rightarrow \infty} b(s)P\bigg(X+Y<-sx\bigg) \leq \lim_{s\rightarrow \infty}
		b(s)\bigg[P\bigg(X<-\frac{sx}{2}\bigg)+P\bigg(Y<-\frac{sx}{2}\bigg) \bigg]
		< \infty.		
		\]
		Hence $X+Y\in \mathbb{V}_b$.
		
		To show closure under scalar multiplication, let $c\in \mathbb{R}$. Note that the case $c=0$ is trivial. For $c\neq 0$, we have
		\begin{eqnarray*}
		\lim_{s\rightarrow \infty} b(s)P(cX>sx) &=& \begin{cases}
			\lim_{s\rightarrow \infty} b(s)P(X>s\frac{x}{c}) & \text{if } c>0\\
			\lim_{s\rightarrow \infty} b(s)P(X<s\frac{x}{c}) & \text{if } c<0		
		\end{cases}, \mbox{ and }\\
		\lim_{s\rightarrow \infty} b(s)P(cX<-sx) &=& \begin{cases}
			\lim_{s\rightarrow \infty} b(s)P(X<-s\frac{x}{c}) & \text{if } c>0\\
			\lim_{s\rightarrow \infty} b(s)P(X>-s\frac{x}{c}) & \text{if } c<0		
		\end{cases},
		\end{eqnarray*}
		and each is finite by virtue of $X\in\mathbb{V}_b$.

	The random variable $X^* \equiv 0$ serves as the additive identity. Furthermore, as $\mathbb{V}_b$ is closed under scalar multiplication, for any $X\in \mathbb{V}_b$ the element $-X$ serves as the additive inverse. The scalar $1 \in \mathbb{R}$ serves as the identity of scalar multiplication. Finally, all associativity, commutativity and distribution laws follow from point-wise operations on (measurable) real-valued functions. Hence $\mathbb{V}_b$ is a real vector space.
	\end{proof}
    
    To define the equivalence relation, we will identify a subspace of $\mathbb{V}_b$ which consists of all random variables whose limiting tail probabilities, when scaled by $b$, approach zero. As this concept is central to our construction, we formalize it with a definition.

    \begin{definition}
		We say an arbitrary random variable $X$ is trivially normalized by a given normalizing function $b$ provided that for all $x>0$
		\[
		\lim_{s\rightarrow \infty} b(s) P(X>sx) = 0, \label{eq:2.1.1} \tag{2.1.1}
		\]
		\[
		\lim_{s\rightarrow \infty} b(s) P(X<-sx) = 0. \label{eq:2.1.2} \tag{2.1.2}
		\]
	\end{definition}

	\begin{definition}
        Given a normalizing function $b$, define ${\cal N}_b$ to be the set of all $X\in \mathbb{V}_b$ such that $X$ is trivially normalized by $b$.
	\end{definition}
    	
	\begin{proposition}
		For a fixed normalizing function $b$, ${\cal N}_b$ is a vector subspace of $\mathbb{V}_b$. 
	\end{proposition}
	
	\begin{proof}
		We only need to show closure under addition and scalar multiplication, as the remaining vector space axioms follow from the fact ${\cal N}_b$ is a subset of  $\mathbb{V}_b$. First, suppose $X,Y\in {\cal N}_b$. Then
		\[
			0 \leq b(s)P(X+Y>sx) \leq b(s)P\bigg(X>\frac{sx}{2}\bigg) + b(s)P\bigg(Y>\frac{sx}{2}\bigg)\rightarrow 0+0 = 0
		\]
		and 
		\[
			0 \leq b(s)P(X+Y<-sx) \leq b(s)P\bigg(X<-\frac{sx}{2}\bigg) + b(s)P\bigg(Y<-\frac{sx}{2}\bigg)\rightarrow 0+0 = 0.	
		\]
		Hence $X+Y \in {\cal N}_b$.
		
		Closure under scalar multiplication follows in the same manner as the previous proof for $\mathbb{V}_b$. Closure under scalar multiplication gives the ordinary additive inverse for an element $X\in {\cal N}_b$. The remaining properties of vector spaces are inherited from $\mathbb{V}_b$. Hence ${\cal N}_b$ is a subspace of $\mathbb{V}_b$.
	\end{proof}
	
	We are now ready to define equivalence classes. 
    Declare two random variables $X,Y \in \mathbb{V}_b$ to be equivalent, denoted $X \sim_{b} Y$, if and only if $X-Y \in {\cal N}_b$. We then form the space of all such equivalence classes.

	\begin{definition}
		For a fixed normalizing function $b$, let $\mathbb{W}_b$ be the quotient space
		\[
		\mathbb{W}_b = \mathbb{V}_b \mathbin / {\cal N}_b.
		\]
        The quotient space $\mathbb{W}_b$ is a vector space by construction.
	\end{definition}
		\begin{remark}
        To motivate the construction of our vector space $\mathbb{W}_b$, consider a naively constructed space of only regularly varying random variables with normalizing function $b$. 
        Such a space will not satisfy the vector space axioms, as it will not be closed under addition. For example, let $b(s)=s^\alpha$, $X\sim \textit{Frechet}(\alpha)$ and define $Y = -X \cdot 1\{X>10^6\}$. Then both $X$ and $Y$ are regularly varying with normalizing function $b$, but $X+Y$ is not; indeed, the sum is a bounded random variable. 
        In order to achieve closure, a vector space must include members with lighter tails, necessitating the inclusion of random variables trivially normalized by $b$, and leading to the vector space of equivalence classes $\mathbb{W}_b$.
	\end{remark}
    
    Theorem \ref{thm:1} below shows that if $X \in \mathbb{W}_b$, then either $X$ is trivially normalized or $X$ is regularly varying in the sense of definition \ref{def:regvar}.

	\begin{theorem} \label{thm:1}
		The non-zero equivalence classes of $\mathbb{W}_b$ are equivalence classes of regularly varying random variables.
	\end{theorem}

	\begin{proof}		
		Let the normalizing function $b(s)$ be $b(s)=s^\alpha L_1(s)$, where $L_1(s)$ is a slowly varying function. Let $X\in \mathbb{V}_b$ be non-trivially normalized by $b$, and denote the survival function of the random variable $|X|$ by $\bar{F}$. We will show that $|X|$ is regularly varying in the classical sense; that is, we will show $lim_{s\rightarrow \infty} \frac{\bar{F}(sx)}{\bar{F}(s)} = x^{-\alpha}$ for $x>0$.
		
	 	Let $\lim_{s\rightarrow \infty} b(s)P(X>sx)>0$  or $\lim_{s\rightarrow \infty} b(s)P(X<-sx)>0$ for at least one $x=x_0\in(0,\infty)$. Then we have $\lim_{s\rightarrow \infty} b(s)\bar{F}(sx_0) = c \in (0,
		\infty)$ for some positive $x_0$. Thus,
		\[
		\bar{F}(sx_0) \sim c(b(s))^{-1} = c s^{-\alpha} L_2(s) = s^{-\alpha} L_3(s) \quad(s\rightarrow \infty)
		\] 
		where $L_2(s) = \frac{1}{L_1(s)}$, and $L_3(s) = \frac{c}{L_1(s)}$ are both slowly varying. Now let $t=sx_0$; then we have
		\[
		\bar{F}(t) \sim \bigg(\frac{t}{x_0}\bigg)^{-\alpha} L_3\bigg(\frac{t}{x_0}\bigg) = t^{-\alpha}L_4(t) \quad(t\rightarrow \infty). \label{eq:2.1.3} \tag{2.1.3}
		\]
		Here, $L_4(t) = x_0^\alpha L_3(\frac{t}{x_0})$ is still slowly varying.
		
		Thus, using relation \eqref{eq:2.1.3}, it follows that for any $x>0$,  $\bar{F}(tx)\sim (xt)^{-\alpha} L_4(tx)$ as $t\rightarrow \infty$. We then have:
		\[
		\frac{ \frac{\bar{F}(tx)}{(tx)^{-\alpha}L_4(tx)} }{ \frac{\bar{F}(t)}{t^{-\alpha}L_4(t)} } \rightarrow 1 \text{ as } (t\rightarrow \infty).
		\]
		Now, since $L_4$ is slowly varying, it must be the case that 
		\[
		\frac{
		\bar{F}(tx)}{\bar{F}(t)} \rightarrow x^{-\alpha}.
		\]
		Thus, $|X|$ is regularly varying in the classical sense. Now, since the tail of $|X|$ is normalizable at $x_0$ by both $b(s)$ and $\bar{F}(s)$, each of $b$ and $\bar{F}$ must have the same asymptotics. It then follows that the tail of $|X|$ must be normalizable by $b(s)$ in general. That is, $b(s)P(|X|>sx) \rightarrow \kappa x^{-\alpha}$ for all $x>0$ and some $\kappa>0$.
		
		Finally, an element $X\in \mathbb{V}_b$ is mapped to the zero equivalence class if and only if $b(s)P(|X|>sx)\rightarrow 0$ for all $x>0$. So if an element is not mapped to the zero equivalence class, there must be at least one $x_0>0$ such that  $b(s)P(|X|>sx_0)\rightarrow c>0$, in which case it is in fact regularly varying. It then follows that $|X|$ must be normalizable by $b(s)$ with some overall scale $\kappa$, and at least one heavy tail. Thus, all non-zero equivalence classes of $\mathbb{W}_b$ are equivalence classes of regularly varying random variables. 
	\end{proof}
    
    Like the $L^p$ spaces, $\mathbb{W}_b$ is a vector space of equivalence classes of random variables, and linear operations in $\mathbb{W}_b$ do not depend on the representatives of the equivalence classes under consideration. That is, linear operations may be performed on equivalence classes by taking representatives of the equivalence classes under consideration, performing the operations on these representatives, and passing back to equivalence classes. 
    
    In contrast to $L^p$ spaces where equivalence is almost sure equality, the equivalence classes in $\mathbb{W}_b$ are more coarse in that there are random variables that are not equal almost surely but are still in the same equivalence class in $\mathbb{W}_b$. To illustrate, suppose $X_1,X_2$ are in the same equivalence class.
    Then a random variable $Z$ belonging to the same equivalence class can be found that satisfies $X_1 = Z + \epsilon_1$ and $X_2 = Z + \epsilon_2$, where $\epsilon_1, \epsilon_2\in {\cal N}_b$.
    The variables $\epsilon_i, i = 1,2$ are light-tailed in the sense that they are trivially normalized by $b(s)$, but can have unbounded support.
    Far enough into the tails of $X_1$ and $X_2$, the light tailed components become negligible and the contributions from $Z$ come to dominate both.
    Thus, the extreme tails of $X_1$ and $X_2$ will exhibit nearly perfect dependence, which will only strengthen the further into the tails one looks. 
    
    In situations where one is concerned with properties of the extreme (often asymptotic) tail of distributions, the mantra is to let the tail speak for itself. The space $\mathbb{W}_b$ and its equivalence classes provide a formal way of doing this. If two random variables in $\mathbb{V}_b$ have nearly identical asymptotic tails, exhibiting the same asymptotic marginal and dependence structure, they are treated as equivalent regardless of how different they appear in the bulk of their distributions. Similarly, if they have different asymptotic tails, either in dependence or marginal structure, they are treated as distinct elements of $\mathbb{W}_b$ even if they look nearly identical in the bulk of their distributions. In this way, the equivalence classes of $\mathbb{W}_b$ are naturally adapted to roles in extreme value analysis.

	\subsection{Asymptotic Dependence Measures} \label{sec:2.2}

	We show how a definition of robust covariance can be combined with the $\alpha$-scale to characterize dependence among pairs of random variables in the space $\mathbb{W}_b$. Since the non-trivial members of our vector space are regularly-varying and tail equivalent, the dependence measure can be viewed as a measure of (asymptotic) dependence among regularly-varying random variables, but importantly, no assumption of multivariate regular variation will be required. Central to our construction is the notion of a scale functional
		
	\begin{definition} \label{def:scalefunctional}
		A scale functional $S(\cdot)$ is a function mapping a set of distribution functions into the nonnegative reals with the property: for $a,b\in \mathbb{R}$ and distribution function $F_X$ of a random variable $X$, we have
		\[
			S(F_{aX+b}) = |a|S(F_X).
		\]
	\end{definition}
	\begin{remark}
		For ease of notation, we will write $S(F_X) = S(X)$, where $X$ has distribution function $F_X$. In the same way, $S(X+Y)$ is shorthand for $S(F_{X+Y})$.
	\end{remark}
	\begin{remark}
		As suggested by the notation in our previous remark, scale functionals can be interpreted as a map from a set of random variables into the nonnegative reals. The map then depends only on the distribution of a random variable. This fact will be used below. 
	\end{remark}

	One approach to robust covariance modeling begins with the following well-known form for the covariance between two random variables $X,Y$: for nonzero $a,b$
	\[
		Cov(X,Y) = \frac{1}{4 a b} \bigg[var(aX+bY) - var(aX-bY) \bigg]. \label{eq:2.2.1} \tag{2.2.1}
	\]
    \cite{HR09} (section 8.2) construct a robust covariance by replacing variance in \ref{eq:2.2.1} with $S^2(\cdot)$, where $S$ is a robust scale functional. In practice, the constants $a,b$ are taken to be robust scale estimates. A re-scaling of \ref{eq:2.2.1} such that the value is restricted to $[-1,1]$ yields a robust correlation coefficient. We follow a similar approach, substituting an asymptotic measure of scale arising from the $\alpha$-scale of elements of $\mathbb{W}_b$ in place of a robust measure of scale.
    
    \begin{definition}
    	Define the function $S:\mathbb{V}_b \rightarrow [0,\infty)$ by
    	\[
    	S(X) := \bigg\{\lim_{s \rightarrow \infty} b(s)P(|X|>s)\bigg\}^{\frac{1}{\alpha}} = \kappa_X^{\frac{1}{\alpha}} \label{eq:2.2.2} \tag{2.2.2}
    	\]
    \end{definition}
    \begin{remark}
		    In Section \ref{sec:3} when we seek to define a norm for $\mathbb{V}_b$, this function will serve as the norm, but will require restrictions on $\alpha$. Here we require no such restriction, and in what is to follow we work with general $\alpha>0$.
	\end{remark}
    \begin{remark}
    	There is slight abuse of notation here, as we have used the same notation $S$ for both the specific scale functional above, and a general scale functional in definition \ref{def:scalefunctional}. However, as the specific scale functional define here is the only one of interest in what is to follow, there should be no confusion. 
    \end{remark}

	Below we will show that $S$ is a scale functional on $\mathbb{V}_b$. We will also show that it can be thought of as a functional on $\mathbb{W}_b$, as it evaluates to the same value for each representative of a given equivalence class in $\mathbb{W}_b$. We begin with a technical lemma.
	
	\begin{lemma} \label{lem:1}
		Let $X\in \mathbb{V}_b$ satisfy $b(s)P(|X|>sx)\rightarrow \kappa_X x^{-\alpha}$ for normalizing function $b$ and each $x>0$. Let $c_{x}(s)$ be an eventually positive function satisfying $c_{x}(s)\rightarrow x$. Then 
		\[
			b(s)P(|X|>s c_{x}(s)) \rightarrow \kappa_X x^{-\alpha}
		\]
	\end{lemma}
    \begin{proof}
    	We focus on the right tail, as the proof is analogous for the left tail. Let $\epsilon>0$, then for some large value $S$, we have $x-\epsilon\leq x(s) \leq x+\epsilon$ whenever $s>S$. We then have 
    	\[
    		P(X>s(x+\epsilon)) \leq P(X>s c_{x}(s)) \leq P(X> s(x-\epsilon)).
    	\]
    	Upon multiplying by $b(s)$ and talking $s\rightarrow \infty$, we have
    	\[
    		\kappa_{+}(x+\epsilon)^{-\alpha} \leq \lim b(s)P(X>s c_{x}(s)) \leq \kappa_{+}(x-\epsilon)^{-\alpha}.
    	\]
    	As $\epsilon>0$ is arbitrary, we have $\lim b(s)P(X>s c_{x}(s))= \kappa_{+} x^{-\alpha}$. A similar argument can be applied to the left tail. Combining these two arguments proves the lemma. 
    \end{proof}

    \begin{proposition}
    	For $\alpha>0$, the function $S(X)$ in (\ref{eq:2.2.2})  is a scale functional on $\mathbb{V}_b$.
    \end{proposition}
    \begin{proof}
    	First we show $S(aX)=|a|S(X)$. Let $X\in \mathbb{V}_b$ have $\alpha$-scale $\kappa_X$, which could be zero, and let $a\in \mathbb{R}\setminus \{0\}$ (the case $a=0$ is trivial). We then have:
    	\[
    	S(aX) = \{\lim_{s\rightarrow \infty} b(s)P(|aX|>s)\}^\frac{1}{\alpha} 
    	= (\kappa_X |a|^{\alpha})^{\frac{1}{\alpha}} 
    	= |a| S(X).
    	\] 
    	Next, we show $S(X+b)=S(X)$. Write
    	\begin{align*}
    		b(s)P(|X+b|>s) &= b(s) \bigg[ P(X>sx-b)+P(X<-sx-b) \bigg] \\
    		&= b(s) \bigg[P\bigg(X>s \bigg(x-\frac{t}{s}\bigg)\bigg)+
    		P\bigg(X>s \bigg(x+\frac{t}{s}\bigg)\bigg) \bigg]
    	\end{align*}
    	Taking limits and applying the arguments in lemma \ref{lem:1} gives $S(X+b)=S(X)$.
		Finally, combining the above two properties, for $a\neq 0$ we have
    	\[
    	S(aX+b) = S(a(X+\frac{b}{a}))=|a|S(X+\frac{a}{b})=|a|S(X).
    	\]  	
    	For $a=0$ the result is trivial.
    \end{proof}
    
    The scale functional $S$ can also be considered as a mapping $S:\mathbb{W}_b \rightarrow [0,\infty)$. By construction, $S$ depends only on the asymptotic tail of a random variable in $\mathbb{V}_b$. It therefore evaluates to the same value for each representative of an equivalence class in $\mathbb{W}_b$. This leads us to form an asymptotic covariance measure by mimicking the formulation of discussed above.
   
    \begin{definition}
    	For $X,Y \in \mathbb{W}_b$, let $a,b$ be arbitrary and define the extremal covariance
    	\[
    	EC(X,Y) = \frac{1}{4 a b} \bigg[S^2(aX+bY) - S^2(aX-bY) \bigg]. \label{eq:2.2.3} \tag{2.2.3}
    	\]
    \end{definition}

    The measure $EC$ is a very general measure of asymptotic dependence, and enjoys considerable flexibility through selecting specific values for $a$ and $b$. However, some general properties are of note. Firstly, if $X$ and $Y$ are asymptotically independent, then $EC(X,Y)$ is equal to $0$. Though the converse is not necessarily true, this is analogous to the case of classical covariance. Additionally, like a classical covariance, $EC(X,Y)$ depends not only on the strength and direction of asymptotic dependence between $X$ and $Y$, but also their scales. In practice, this latter property frustrates interpretation of $EC(X,Y)$. We solve this by defining an more interpretable extremal correlation. 

 	\begin{definition}
    For $X,Y \in \mathbb{W}_b$, let $a=S(X)^{-1},b=S(Y)^{-1}$ define the extremal correlation
 	    \[
    	ER(X,Y) = \frac{S^2(aX+bY) - S^2(aX-bY)}{S^2(aX+bY) + S^2(aX-bY)}. \label{eq:2.2.4} \tag{2.2.4}
    	\]
 	\end{definition}
    
    The restriction of $a,b$ in the definition of the extremal correlation $ER$ ensures that it takes values in $[-1,1]$ and will thus be more interpretable. In particular, if $X$ and $Y$ are multiples of the same equivalence class, $X=cY$ for some $c\neq 0$, we have the desirable property that $ER(X,Y) \in \{-1,1\}$ depending on the sign of $c$. It also inherits from $EC$ the property that if $X$ and $Y$ are asymptotically independent, then $ER(X,Y)$ is equal to $0$. See \ref{app:A1} for details.

	\section{A Normed Vector Space for Regular Variation} \label{sec:3}
	
	\subsection{A Natural Norm for $\mathbb{W}_b$} \label{sec:3.1}
	
	We now endow the space $\mathbb{W}_b$ with additional structure, in particular with a norm. To do this, we instead endow the base space $\mathbb{V}_b$ with a \textit{seminorm}, $||\cdot||_b$. We then use this seminorm to define the same null space as we defined above, ${\cal N}_b = \{X\in \mathbb{V}_b:  ||X||_b = 0\}$. Then, just as above, we take a quotient to get a vector space $\mathbb{W}_b$. Using this method of construction, we obtain the same space $\mathbb{W}_b$ as above, for the space ${\cal N}_b$ is the same; however, as the subspace ${\cal N}_b$ is now defined using a seminorm, this gives rise to a natural norm on the quotient space $\mathbb{W}_b$.
	
	The natural candidate for a seminorm is given by the scale functional (\ref{eq:2.2.2}):
	\[
	||X||_b = S(X) = \bigg\{\lim_{s \rightarrow \infty} b(s)P(|X|>s)\bigg\}^{\frac{1}{\alpha}} = \kappa_X^{\frac{1}{\alpha}}.
	\]
	Showing that $||\cdot||_b$ is a seminorm is not an easy task. While non-negativity and absolute homogeneity of $||\cdot||_b$ are easily shown, showing the triangle inequality is satisfied is not simple. After some technical lemmas, we show in theorem \ref{thm:2} below that if $\alpha>1$, the triangle inequality is satisfied and that $||\cdot||_b$ is a seminorm. 
	
	Before we present some technical lemmas, we recall some basic definitions and outline the direction of our proof. First, recall that two random variables $Y_1,Y_2$ are comonotone if one is a monotone function of the other. This is typically expressed as $(Y_1,Y_2) = (F_{Y_1}^{-1}(U), F_{Y_2}^{-1}(U))$, where $U$ is a uniform random variable. To begin building some intuition as to how the proof will proceed, restrict attention to positive elements $X_1, X_2 \in \mathbb{V}_b$. In this case, it would seem the stronger the positive dependence between $X_1$ and $X_2$, the greater the $\alpha$-scale of their sum. As comonotonicity represents a very strong form of of positive dependence, one might guess that the sum of comonotone random variables gives an upper bound for the $\alpha$-scale of sums with the same marginals. Our first lemma confirms this intuition. 
	
	The proof of this lemma utilizes convex order, a useful characterization of which we recall here. A random variable $X$ is said to be less than a random variable $Y$ in convex order provided that $E[X] = E[Y]$ and 
	\[
		\int_x^\infty \bar{F}_X(u) du \leq \int_x^\infty \bar{F}_Y(u)du \quad \forall x.
	\]
	Note this requires the tail index $\alpha$ be greater 1 for regularly varying $X,Y$. For further details, see \cite{SS07}. In addition to providing an upper bound on $\alpha$-scales, comonotonicity has an additional use. Under comonotonicity, quantile functions of sums become additive, and so in the limit, $\alpha$-scales become additive as well \cite{KDVGD02}. This additivity property, coupled with the upper bound on $\alpha$-scales, afforded by comonotonicity is ultimately the simplifying device that makes a proof of the triangle inequality manageable.
	
	\begin{lemma} \label{lem:upperbound}
		Let $X_1, X_2 \geq 0$ be regularly varying with normalizing function $b(s) = s^\alpha L_b(s)$ with $\alpha >1$. Denote their marginals by $F_{X_1}$ and $F_{X_2}$. Furthermore, let $(Y_1, Y_2)$ be comonotonic with the same marginals as $(X_1, X_2)$. Then, $\kappa_{X_1+X_2}^{\frac{1}{\alpha}} \leq \kappa_{Y_1+Y_2}^{\frac{1}{\alpha}}$.  
	\end{lemma}
	\begin{proof}
		As we assumed $\alpha >1$, expectations exist, allowing the use of convex ordering. We denote the convex order by $\leq_{cx}$. It is well known that for $(X_1, X_2)$ and $(Y_1,Y_2)$ with $Y_i =_d X_i$ and $(Y_1,Y_2)$ comonotone, we have the relation
		\[
		X:= X_1+X_2 \leq_{cx} Y_1+Y_2 =: Y.
		\]
		That is, the comonotonic dependence structure gives the largest convex sum (see \cite{KDVGD02}). Now, denote the marginal of $X=X_1+X_2$ by $F_X$ and similarly for $Y=Y_1+Y_2$. We then have 
		\[
		\int_s^\infty \bar{F}_X(t)dt \leq \int_s^\infty \bar{F}_Y(t)dt  \quad{\forall s}. \label{eq:3.1.1} \tag{3.1.1}
		\]
		As $\alpha>1$, both integrals exist. Since both $\bar{F}_X$ and $\bar{F}_Y$ are regularly varying, the asymptotics for these integrals are well understood. Indeed, it is known that as $x\rightarrow \infty$ 
		\[
		\int_x^\infty t^{-\alpha} L(t)dt \sim \frac{x^{-\alpha+1}L(x)}{\alpha-1}.
		\]
		
		Hence, rather than $b(s)$, we ought to scale by $\frac{b(s)}{s}$ in \eqref{eq:3.1.1} above. Doing so and taking limits gives
		
		\[
		\lim_{s\rightarrow \infty}\frac{b(s)}{s}\int_s^\infty \bar{F}_X(t)dt \leq \lim_{s\rightarrow \infty}\frac{b(s)}{s}\int_s^\infty \bar{F}_Y(t)dt,
		\]
		
		\[
		\lim_{s\rightarrow \infty} \bigg\{\frac{s^{\alpha}L_b(s)}{s} \frac{s^{-\alpha +1}L_X(s)}{\alpha-1} \bigg\} \leq \lim_{s\rightarrow \infty} \bigg\{\frac{s^{\alpha}L_b(s)}{s} \frac{s^{-\alpha +1}L_Y(s)}{\alpha-1} \bigg\},
		\]
		
		\[
		\lim_{s\rightarrow \infty} L_b(s)L_X(s) \leq \lim_{s\rightarrow \infty} L_b(s)L_Y(s),
		\]
		
		\[
		\kappa_X \leq \kappa_Y.
		\]
		
		Now, recalling that $X=X_1+X_2$ and $Y=Y_1+Y_2$, taking $\frac{1}{\alpha}$ roots gives
		
		\[
		\kappa_{X_1+X_2}^{\frac{1}{\alpha}} \leq \kappa_{Y_1+Y_2}^{\frac{1}{\alpha}}
		\]
		as desired.
	\end{proof}

	Lemma \ref{lem:upperbound} simplifies the analyses of the triangle inequality for $||\cdot||_b$, as it reduces the analysis to comonotone random variables. As the $\alpha$-scales of elements of $\mathbb{V}_b$ depend only on the marginal distribution, if the triangle inequality holds for comonotone elements, lemma \ref{lem:upperbound} shows that the triangle inequality will hold in general.
	
	To show the triangle inequality for comonotone random variables, we make use of the additivity of the quantile function of sums under comonotonicity. This requires a way to extract the $\alpha$-scale of a random variable $X$ with c.d.f. $F_X (x)$ from its quantile function $Q_X(p)$. Equivalently, a more natural approach is to work with tail quantile function: $U_X(t):=Q_X(1-\frac1t)$ for $t\geq 1$. Additivity of the quantile function is easily seen to translate to the tail quantile function. The reason for the jump to tail quantile functions are the more natural asymptotics enjoyed by the tail quantile function. As the tail quantile function is the asymptotic inverse of $(1-F(t))^{-1}$, it has the well known asymptotic representation
	\[
	U_X(t) \sim t^{\frac{1}{\alpha}} (L_X^{-\frac{1}{\alpha}})^*(t^{\frac{1}{\alpha}}) \quad(t\rightarrow \infty).
	\] 
	Here $*$ denotes de Brujin conjugacy (cf \cite{BGST05}).
	
	The next few results establish the proper modification to our normalizing function $b$ in order to normalize the tail quantile function and recover $\alpha$-scale information.
	
	\begin{lemma} \label{lem:asconj}
		Let $X \geq 0 $ be a regularly varying random variable with survival function $\bar{F} = x^{-\alpha}L_X(x)$ and scale $\kappa_X$ when normalized by $b(s)=s^\alpha L_b(s)$. Then 
		\[
		(L_X^{\frac{-1}{\alpha}})^{*}(t^{\frac{1}{\alpha}}) \sim \kappa_X^{\frac{1}{\alpha}} (L_b^{\frac{1}{\alpha}})^{*}(t^{\frac{1}{\alpha}}) (t\rightarrow \infty),
		\] 
		where $*$ denotes de Brujin conjugacy. 
	\end{lemma}
	\begin{proof}
		It is well know that if $L(x)\sim U(x)$ as $x\rightarrow \infty$ with $L$ and $U$ slowly varying at infinity, then we have the conjugacy relation $L^*(x) \sim U^*(x)$ as $x\rightarrow \infty$. 
		
		Applying this to the functions $L_X$ and $L_b$ above, we have
		\begin{align*}
			\kappa_x^{-1} L_b(s) L_X(s) &\sim 1 \quad(s\rightarrow \infty), \\
			\kappa_x^{-\frac{1}{\alpha}} L_b(s)^{\frac{1}{\alpha}} L_X(s)^{\frac{1}{\alpha}} &\sim 1 \quad(s\rightarrow \infty), \\		
			\kappa_x^{-\frac{1}{\alpha}} L_b(s)^{\frac{1}{\alpha}} &\sim L_X(s)^{-\frac{1}{\alpha}} \quad(s\rightarrow \infty). \\			
		\end{align*}
		Thus, utilizing the conjugacy relation above, we have 
		$$
		\big(L_X^{-\frac{1}{\alpha}}\big)^{*}(s) \sim \bigg(\kappa_x^{-\frac{1}{\alpha}} L_b(s)^{\frac{1}{\alpha}}\bigg)^{*} \sim \kappa_x^{\frac{1}{\alpha}} \big(L_b^{\frac{1}{\alpha}}\big)^{*}(s) \quad(s\rightarrow \infty). 
		$$
		
		Now, setting $s = t^{\frac{1}{\alpha}}$ gives
		$$
		\big(L_X^{-\frac{1}{\alpha}}\big)^{*}(t^{\frac{1}{\alpha}}) \sim \kappa_x^{\frac{1}{\alpha}} \big(L_b^{\frac{1}{\alpha}}\big)^{*}(t^{\frac{1}{\alpha}}) \quad(t\rightarrow \infty), 
		$$
		as desired.
	\end{proof}
	
	The following proposition uses lemma \ref{lem:asconj} to demonstrate how $\alpha$-scale information can be recovered from $U_X(t)$.
	
	\begin{proposition} 
		Suppose $X\geq 0$ is a regularly varying random variable with survival function $\bar{F} = x^{-\alpha} L_X(x)$, and scale $\kappa_X$ when normalized by $b(s)=s^\alpha L_b(s)$. Furthermore, let $Q_X(p)$ be the quantile function of $X$. Define the tail quantile function $U_X(t) = Q_X(1-\frac1t)$ for $t\geq1$. Then 
		\[
		b^{*}(t) U_X(t) \sim \kappa^{\frac{1}{\alpha}},
		\]
		where $b^{*}(t) = t^{-\frac{1}{\alpha}} \bigg[(L_b^{\frac{1}{\alpha}})^{*}(t^{\frac{1}{\alpha}}) \bigg]^{-1}$.
	\end{proposition}
	\begin{proof}
		Observe
		\begin{align*}
	 	 \lim_{t\rightarrow \infty} b^{*}(t) U_X(t) &= \lim_{t\rightarrow \infty} t^{-\frac{1}{\alpha}}  [(L_b^{\frac{1}{\alpha}})^{*}(t^{\frac{1}{\alpha}})]^{-1} \cdot t^{\frac{1}{\alpha}} (L_X^{\frac{-1}{\alpha}})^{*}(t^{\frac{1}{\alpha}}) \\
	 	 &= \lim_{t\rightarrow \infty} \frac{(L_X^{-\frac{1}{\alpha}})^{*}(t^{\frac{1}{\alpha}})} {(L_b^{\frac{1}{\alpha}})^{*}(t^{\frac{1}{\alpha}})} \\
	 	 &= \kappa_{X}^{\frac{1}{\alpha}}.
	 	\end{align*}

		Here we have used the asymptotic expression for $U_X$ in the first equality, and lemma \ref{lem:asconj} in the second.
	\end{proof}

	This final lemma shows that not only does the triangle inequality for $||\cdot||_b$ hold for positive comonotonic elements of $\mathbb{V}_b$, but it is in fact an equality. 

	\begin{lemma} \label{lem:comequal}
	Let $Y_1, Y_2 \geq 0$ be comonotone random variables with regularly varying marginals, with scales $\kappa_{Y_1}, \kappa_{Y_2}$ when normalized by $b(s) = s^{\alpha}L_b(s)$. Then 
	\[
	\kappa_{Y_1+Y_2}^{\frac{1}{\alpha}} = \kappa_{Y_1}^{\frac{1}{\alpha}} +\kappa_{Y_2}^{\frac{1}{\alpha}}.
	\]
	\end{lemma}
	\begin{proof}
		Since $Y_1$ and $Y_2$ are comonotone, the quantile function of the sum, $Q_{Y_1+Y_2}(p)$ is additive. Rewriting $p=1-\frac1t$, we can recover the tail quantile function, giving
		\[
		U_{Y_1+Y_2}(t) = U_{Y_1}(t) + U_{Y_2}(t).
		\]
		
		Now, using the proposition above, multiplying by $b^*(t)$ and letting $t\rightarrow\infty$ we have
		\[
		\kappa_{Y_1+Y_2}^{\frac{1}{\alpha}} = \kappa_{Y_1}^{\frac{1}{\alpha}} + \kappa_{Y_2}^{\frac{1}{\alpha}},
		\]
		as desired.
	\end{proof}
	
	We are now ready to present the main theorem of this section. The proof is a straightforward application of the lemmas above.
	
	\begin{theorem} \label{thm:2}
		Let $X\in \mathbb{V}_b$ and define 
		\[
		||X||_b = \kappa_X^{\frac{1}{\alpha}} = \bigg\{\lim_{s\rightarrow \infty} b(s)P(|X|>s) \bigg\}^{\frac{1}{\alpha}}.
		\]
		Then $||X||_b$ defines a seminorm on $\mathbb{V}_b$ for $\alpha>1$.
	\end{theorem}
	\begin{proof}
		First, note that $||\cdot||_b$ is clearly non-negative for any $\alpha>0$. Next, we show that  $||\cdot||_b$ is absolutely homogeneous. Let $X\in \mathbb{V}_b$ have $\alpha$-scale $\kappa_X$, which could be zero, and let $a\in \mathbb{R}\setminus \{0\}$ (the case $a=0$ is trivial). We then have:
		\[
		||aX||_b = \{\lim_{s\rightarrow \infty} b(s)P(|aX|>s)\}^\frac{1}{\alpha} \\
		= (\kappa_X |a|^{\alpha})^{\frac{1}{\alpha}} \\
		= |a| ||X||_b,
		\] 
		as desired. (Note, this requires only that $\alpha>0$.)
		
		All that remains is to show that the triangle inequality is satisfied. First note it is sufficient to consider only positive random variables. For, if $X_1, X_2 \in \mathbb{V}_b$ are arbitrary, we have
		\[
		P(|X_1+X_2|>s) \leq P(|X_1| + |X_2| >s) = P(||X_1|+|X_2||>s),
		\]
		which implies that $\kappa_{X_1+X_2}^{\frac{1}{\alpha}} \leq \kappa_{|X_1|+|X_2|}^{\frac{1}{\alpha}}$. Now, since $\kappa_X = \kappa_{|X|}$, we must simply show that $\kappa_{|X_1|+|X_2|}^{\frac{1}{\alpha}} \leq \kappa_{|X_2|}^{\frac{1}{\alpha}} + \kappa_{|X_2|}^{\frac{1}{\alpha}}$. From this, the triangle inequality will follow. Hence we may restrict our attention to positive random variables. The proof is now a simple application of the above lemmas.
		
		Thus, without loss of generality, let $X_1, X_2 \in \mathbb{V}_b$ be positive random variables with marginals $F_{X_1}, F_{X_2}$ respectively. Furthermore, $Y_1,Y_2$ be the comonotone version of $X_1,X_2$, that is $F_{Y_1} = F_{X_1}, F_{Y_2}=F_{X_2}$ and $(Y_1,Y_2)$ have the comonotone dependence structure. 
		
		By lemma \ref{lem:upperbound}, we can bound the scale of $X_1+X_2$ by that of $Y_1+Y_2$, giving
		\[
		\kappa_{X_1+X_2}^{\frac{1}{\alpha}} \leq \kappa_{Y_1+Y_2}^{\frac{1}{\alpha}}. \label{eq:3.1.2} \tag{3.1.2}
		\]
		Now, an application of lemma \ref{lem:comequal} shows that 
		\[
		\kappa_{Y_1+Y_2}^{\frac{1}{\alpha}} = \kappa_{Y_1}^{\frac{1}{\alpha}}+\kappa_{Y_2}^{\frac{1}{\alpha}}. \label{eq:3.1.3} \tag{3.1.3}
		\]
		Putting both \eqref{eq:3.1.2} and \eqref{eq:3.1.3} together, we have
		\[
		\kappa_{X_1+X_2}^{\frac{1}{\alpha}} \leq \kappa_{Y_1}^{\frac{1}{\alpha}}+\kappa_{Y_2}^{\frac{1}{\alpha}} = \kappa_{X_1}^{\frac{1}{\alpha}}+\kappa_{X_2}^{\frac{1}{\alpha}},
		\]
		where the last inequality follows from the fact that $X_i$ and $Y_i$ have identical marginals, and hence identical $\alpha$-scales. Thus, the triangle inequality is satisfied and the proof is complete.
	\end{proof}
	\begin{remark}
		If $\alpha < 1$, the triangle inequality will fail. Most literature on this topic is in the context of sub or super-additivity of Value-at-Risk (Var) for large investment portfolios. \cite{ME12} discuss the sub or super-additivity of VaR and show, using multivariate regular variation, that VaR is super-additive in the case $\alpha<1$. From the super-additivity of VaR, counterexamples are seen to be abundant. 
	\end{remark}
	\begin{remark}
		In the case of $\alpha=1$, the case is much more complicated. \cite{ME12} found the VaR to be additive if restricted to positive random variables, and subadditive in the general case. However, their approach relies on multivariate regular variation, and doesn't rule out the possibility that subadditivity will fail for two random variables that are not jointly regularly varying. This edge case is thus an open problem in general.
	\end{remark} 
	
	Finally, recall the earlier construction of $\mathbb{W}_b$ as a quotient space. The subspace ${\cal N}_b$ of elements with zero $\alpha$-scale is also the closed subspace of all $X\in \mathbb{V}_b$ with seminorm $||X||_b = 0$. Thus, using ${\cal N}_b = \{X\in \mathbb{V}_b: ||X||_b = 0\}$, we may form $\mathbb{W}_b = \mathbb{V}_b \mathbin / {\cal N}_b$ as before. In using this quotient construction, we obtain a normed vector space consisting of equivalence classes of tail equivalent regularly varying random variables. We thus have the following theorem, which follows immediately in light of the above discussion.
 
 	\begin{theorem} \label{thm:normspace}
	For a given normalizing function $b(s) = s^\alpha L_b(s)$ with $\alpha>1$, $(\mathbb{W}_b, ||\cdot||_b)$ is a normed vector space.
	\end{theorem}
	
	\subsection{Remarks on Convergence in $\mathbb{W}_b$} \label{sec:3.2}
	
	A natural question to ask is whether the space $\mathbb{W}_b$ is complete. At this point, the issue of completeness is unresolved.Nevertheless, some remarks on the difficulty of demonstrating completeness are in order, and best illustrated by a few examples. As the elements of $\mathbb{W}_b$ are rather strange, we rely on the following fact: the normed space$(\mathbb{W}_b, ||\cdot||_b)$ is complete if and only if $(\mathbb{V}_b, ||\cdot||_b)$ is complete. Thinking about completeness of $\mathbb{W}_b$ can thus be accomplished by thinking of random variables proper.
	
	\begin{example}
		Let $Z_1, Z_2,...$ be a sequence of i.i.d. standard normal random variables in $\mathbb{V}_b$. Consider the sequence of partial sums $S_n = \sum_{i=1}^n Z_i$. Under ordinary modes of convergence considered in probability theory, the sequence would not converge (at least without proper normalization in some cases). Under $||\cdot||_b$, the sequence of partial sums converges to any element in ${\cal N}_b$. This is a sequence of zero elements in $\mathbb{W}_b$.
	\end{example}
	
	The i.i.d. assumption can be replaced with an arbitrary dependence structure and the result is unchanged.
	
	\begin{example}
		Let $X$ be a Pareto random variables with tail index $\alpha>1$, and let $a_n\rightarrow \infty$ be a positive sequence diverging to infinity. Define the inverse truncated series by $Y_n = X 1_{\{X>a_n\}}$, which keeps only the tails at higher and higher levels. This sequence is indeed convergent, and is in fact a constant sequence. One may take a limit to be $X$, for example. 
	\end{example}
		
	This last example illustrates the difference between $||\cdot||_b$-convergence, and other modes of convergence typically considered in probability theory. If the sequence $a_n$ is taken to diverge fast enough, the sequence $Y_n$ can be made to converge in probability to 0. In such an event, there is a subsequence $Y_{n_k}$ converging almost surely to zero. Thus, $||\cdot||_b$-convergence is not necessarily related to either mode of converge frequently encountered in probability theory.
	
	Indeed, a kind of converse is possible. Let $\{X_n\}$ be a sequence of tail equivalent regularly varying random variables, and suppose it diverges in the $||\cdot||_b$ sense. A similar truncation to that considered in example 2 can be performed and similar statements concerning convergence in probability and almost surely can be achieved if the truncation level $a_n$ is carefully chosen. This further demonstrates the disconnect between $||\cdot||_b$-convergence and more familiar forms of convergence.

	\section{Modeling and The Tail Pairwise Dependence Measure} \label{sec:4}

    Even though the issue of completeness of $\mathbb{W}_b$ is unresolved (for now), later in this section we will focus attention on finite dimensional subspaces whose completeness will allow us to utilize the structure of $\mathbb{W}_b$ for modeling purposes.  We first discuss when and if $\mathbb{W}_b$ has an inner product.

	\subsection{The Tail Pairwise Dependence Measure as an Inner Product} \label{sec:4.2}
	
	In the presence of multivariate regular variation, the tail pairwise dependence measure can be viewed as providing a specific kind of structure on subspaces of $\mathbb{W}_b$. In particular, with some subtlety, it can be viewed as an inner product when the tail index is equal to 2. The content of this section can be viewed as continuing the discussion of dependence measures in section \ref{sec:2.2}, but does not use any results from that section. 
	
	First, we recall some familiar notions from inner product spaces. For a vector space $V$, denote an inner product on $V$ by $<\cdot,\cdot>$ and the norm induced by the inner product by  $||x||:=\sqrt{<x,x>}$. It is well known that in any inner product space, the parallelogram law holds: for all $x,y\in V$,
	\[
		||x+y||^2 + ||x-y||^2 = 2||x||^2 + 2||y||^2.
	\] 
	A converse is also true: Let $(V,||\cdot||)$ be any normed space. If $||\cdot||$ satisfies the parallelogram law, then there is an inner product $<\cdot, \cdot>$ on $V$ that induces the norm $||\cdot||$. We also recall the polarization identity. There are several equivalent forms of the polarization identity, so we present the form used below. If a norm $||\cdot||$ is induced by an inner product, then the inner product is given by
	\[
		<x,y> = \frac14 (||x+y||^2 - ||x-y||^2).
	\]
	
	The main theorem of this section shows that the parallelogram law holds for elements of $\mathbb{W}_b$ under the added assumption of multivariate regular variation. Before we present the main theorem, we state a familiar but useful lemma.
	
	\begin{lemma}
		Let $X_1,X_2$ be jointly regularly varying with exponent measures given by $\nu_{\bm{X}}$ and $H_{\bm{X}}$ respectively. Define the polar transformation $R=\sqrt{X_1^2 + X_2^2}$ and $\bm{\omega} = (\frac{X_1}{R}, \frac{X_2}{R})$. Then, for sets of the form $D_{r,\omega} = \{(R,\bm{\omega}) \in [0,\infty) \times \mathbb{S}^1 : R>f(\bm{\omega}), \bm{\omega}\in B\}$, positive function $f>0$, we have 
		\begin{equation*}
			\nu_{\bm{X}}(D_{r,\omega}) = \int_B \big(f(\bm{\omega}) \big)^{-\alpha} dH_{\bm{X}}(\bm{\omega}).
		\end{equation*}
	\end{lemma}
	\begin{proof}
		The integral expression follows immediately from the differential expression for $\nu_{\bm{X}}$: 
		\begin{equation*}
			d\nu_{\bm{X}} = \alpha r^{-(\alpha+1)}drdH_{\bm{X}}(\bm{\omega}).
		\end{equation*}
		
	We have 
		\begin{align*}
			\nu_{\bm{X}}(D_{r,\omega}) &= \int_{D_{r,\omega}} d\nu_{\bm{X}} \\
			&= \int_{\bm{\omega}\in B} \int_{\{r>f(\bm{\omega})\}} \alpha r^{-(\alpha+1)}drdH({\bm{\omega}}) \\
			&= \int_{\bm{\omega}\in B} \big(f(\bm{\omega})\big)^{-\alpha} dH_{\bm{X}}(\bm{\omega}).
		\end{align*}
	\end{proof}

	\begin{theorem} \label{thm:4}
		Let $X_1, X_2 \in \mathbb{V}_b \setminus {\cal N}_b$ with $\alpha = 2$. Furthermore, let them be jointly regularly varying with exponent measure $\nu_{\bm{X}}$, and angular measure $H_{\bm{X}}$ under the polar transformation. Then $||\cdot||_b$ satisfies the parallelogram law. 
	\end{theorem}
	\begin{proof}
		For now, we treat the tail index $\alpha$ as arbitrary. Define the sets
        \begin{align*}
			C_1 = \{(x_1, x_2) \in \mathbb{R}^2: |x_1+x_2|>1 \}, \\
			C_1^+ = \{(x_1, x_2) \in \mathbb{R}^2: x_1+x_2>1 \}, \\
			C_1^- = \{(x_1, x_2) \in \mathbb{R}^2: x_1+x_2<-1 \}.
		\end{align*}
        Then,
		\begin{align*}
			||X_1+X_2||_b^\alpha &= \lim_{s\rightarrow \infty} b(s)P(|X_1+X_2|>s) \\
			&= \lim_{s\rightarrow \infty} b(s)P\bigg(\frac{1}{s}(X_1, X_2)\in C_1\bigg) \\
			&= \nu_{\bm{X}}(C_1) = \nu_{\bm{X}}(C_1^+) + \nu_{\bm{X}}(C_1^-).
		\end{align*} 
		
		Switching to polar coordinates with $R = \sqrt{X_1^2 + X_2^2}$ and $\bm{\omega} = (\frac{X_1}{R}, \frac{X_2}{R})$, we can rewrite the sets $C_1^+, C_1^-$ as follows:
		
		\begin{align*}
			C_1^+ = \{(R, \bm{\omega}) \in [0, \infty) \times \mathbb{S}^1: R (\omega_1+\omega_2)>1 \}, \\
			C_1^- = \{(R, \bm{\omega}) \in [0, \infty) \times \mathbb{S}^1: R (\omega_1+\omega_2)<-1 \}.
		\end{align*}
		
		To compute $\nu_X(C_1^+) + \nu_X(C_1^-)$ using the lemma above, we rewrite the sets $C_1^+, C_1^-$ in a more manageable form. Define the sets $A_1 = \{\bm{\omega}\in \mathbb{S}^1: \omega_1+\omega_2 >0 \}$ and $A_2 = \{\bm{\omega}\in \mathbb{S}^1: \omega_1+\omega_2 < 0 \}$. Then
		\begin{align*}
			C_1^+ = \{(R, \bm{\omega}) \in [0, \infty) \times \mathbb{S}^1: R> \frac{1}{(\omega_1+\omega_2)}, \bm{\omega}\in A_1 \}, \\
			C_1^- = \{(R, \bm{\omega}) \in [0, \infty) \times \mathbb{S}^1: R> \frac{1}{-(\omega_1+\omega_2)}, \bm{\omega}\in A_2 \}.
		\end{align*}
		Utilizing the lemma above, we have
		
		\begin{align*}
			||X_1+X_2||_b^\alpha &= \int_{A_1} (\omega_1+\omega_2)^\alpha dH_{\bm{X}}(\bm{\omega}) + \int_{A_2} (-\omega_1-\omega_2)^\alpha dH_{\bm{X}}(\bm{\omega}).
		\end{align*}
		Using the same polar transformation, similar reasoning is used to compute $||X_1-X_2||_b^\alpha$. We have
		\begin{align*}
			||X_1-X_2||_b^\alpha &= \lim_{s \rightarrow \infty}b(s)P(|X_1-X_2|>s) \\
			&= \lim_{s \rightarrow \infty}b(s)P\bigg(\frac{1}{s} (X_1, X_2)\in C_2^+ \cup C_2^-\bigg)\\
			&= \nu_{\bm{X}}(C_2^+ \cup C_2^-).
		\end{align*}
		Just as above, the sets $C_2^+$ and $C_2^-$ are given by
		\begin{align*}
			C_2^+ = \{(R, \bm{\omega}) \in [0, \infty) \times \mathbb{S}^1: R> \frac{1}{(\omega_1-\omega_2)}, \bm{\omega}\in A_3 \}, \\
			C_2^- = \{(R, \bm{\omega}) \in [0, \infty) \times \mathbb{S}^1: R> \frac{1}{-(\omega_1-\omega_2)}, \bm{\omega}\in A_4 \}.
		\end{align*}
		In defining these sets, we have further defined $A_3 = \{\bm{\omega}\in \mathbb{S}^1: \omega_1-\omega_2 >0 \}$ and $A_4 = \{\bm{\omega}\in \mathbb{S}^1: \omega_1-\omega_2 < 0 \}$.
		Another application of the above lemma then gives 
		\begin{align*}
			||X_1-X_2||_b^\alpha &= \int_{A_3} (\omega_1-\omega_2)^\alpha dH_{\bm{X}}(\bm{\omega}) + \int_{A_4} (-\omega_1+\omega_2)^\alpha dH_{\bm{X}}(\bm{\omega}).
		\end{align*}
		Thus, we have 
		\begin{align*}
			||X_1+X_2||_b^\alpha+||X_1-X_2||_b^\alpha &=\int_{A_1} (\omega_1+\omega_2)^\alpha dH_{\bm{X}}(\bm{\omega}) + \int_{A_2} (-\omega_1-\omega_2)^\alpha dH_{\bm{X}}(\bm{\omega})\\ 
			&+ \int_{A_3} (\omega_1-\omega_2)^\alpha dH_{\bm{X}}(\bm{\omega}) + \int_{A_4} (-\omega_1+\omega_2)^\alpha dH_{\bm{X}}(\bm{\omega}). \label{eq:4.2.1}\tag{4.2.1}
		\end{align*}

		
		We now restrict our attention to the case $\alpha = 2$. Further, denote the points on the 45 degree line in each of the sets $A_1$ through $A_4$ as follows: $P_1=(\frac{\sqrt2}{2},\frac{\sqrt2}{2}), P_2=(\frac{-\sqrt2}{2},\frac{-\sqrt2}{2}), P_3=(\frac{\sqrt2}{2},\frac{-\sqrt2}{2}), P_4=(\frac{-\sqrt2}{2},\frac{\sqrt2}{2})$. We may simplify the relation \ref{eq:4.2.1} above as follows.
		
		\begin{align*}
		||X_1+X_2||_b^2+||X_1-X_2||_b^2 &=\int_{A_1\cap A_3} (\omega_1+\omega_2)^2 +(\omega_1-\omega_2)^2 dH_{\bm{X}}(\bm{\omega})\\ 
		&+ \int_{A_1\cap A_4} (\omega_1+\omega_2)^2 +(-\omega_1-\omega_2)^2 dH_{\bm{X}}(\bm{\omega}) \\
		&+ \int_{A_2\cap A_3} (-\omega_1-\omega_2)^2 +(\omega_1-\omega_2)^2 dH_{\bm{X}}(\bm{\omega}) \\
		&+ \int_{A_2\cap A_4} (-\omega_1-\omega_2)^2 +(-\omega_1+\omega_2)^2 dH_{\bm{X}}(\bm{\omega}) \\
		&+ \int_{P_1} (\omega_1+\omega_2)^2 dH_{\bm{X}}(\bm{\omega}) +\int_{P_2} (-\omega_1-\omega_2)^2 dH_{\bm{X}}(\bm{\omega}) \\
		&+\int_{P_3} (\omega_1-\omega_2)^2 dH_{\bm{X}}(\bm{\omega})+\int_{P_4} (-\omega_1+\omega_2)^2 dH_{\bm{X}}(\bm{\omega}).
		\end{align*}
		Note that each integral over $P_i$ is equal to $\int_{P_i} 2dH_{\bm{X}}(\bm{\omega})$, which in turn is equal to $\int_{P_i} 2\omega_1^2 + 2\omega_2^2dH_{\bm{X}}(\bm{\omega})$. Thus, expanding the squares in the above integrals and simplifying, we have:
		
		\begin{align*}
			&= \int_{\mathbb{S}^1 \setminus \cup P_i} 2\omega_1^2+2\omega_2^2 dH_{\bm{X}}(\bm{\omega}) + \int_{\cup P_i} 2\omega_1^2+2\omega_2^2 dH_{\bm{X}}(\bm{\omega}) \\
			&= \int_{\mathbb{S}^1} 2\omega_1^2+2\omega_2^2 dH_{\bm{X}}(\bm{\omega}) \\ 
			&= 2||X_1||_b^2 + 2||X_2||_b^2 
		\end{align*}
		Thus, the claim has been shown.
	\end{proof}

	If $\alpha=2$ and joint regular variation holds, the norm $||\cdot||_b$ satisfies the parallelogram law, and thus is induced by an inner product. The requirement of joint regular variation limits the space on which the inner product can be defined, as joint regular variation does not hold for all elements $X,Y\in \mathbb{W}_b$. Nevertheless, theorem \ref{thm:4} makes no assumption about the dimension of the underlying subspace.  For any subspace of $\mathbb{W}_b$, as long as any two elements in the space are jointly regularly varying, these results apply. We further discuss some examples below. The following corollary provides a connection between this inner product and the tail pairwise dependence measure, which is useful for applied problems and calculations.
	
	\begin{corollary} \label{cor:1}
		For $\alpha=2$, the inner product inducing $||\cdot||_b$ is given by the tail pairwise dependence measure.
	\end{corollary}
	\begin{proof}
		The proof is a simple investigation of the polarization identities. We denote the function inducing $||\cdot||_b$ by $<\cdot, \cdot>$. Although any of the polarization identities can be used, we will use the following form:
		
		\begin{equation*}
			<x,y> = \frac14 \big(||x+y||^2 - ||x-y||^2\big).
		\end{equation*}

		Let $X_1$ and $X_2$ be jointly regularly varying with tail index $\alpha=2$ and angular measure $H_{\bm{X}}$. We employ the same transformation as in the above theorem and continue to write the points $P_1=(\frac{\sqrt2}{2},\frac{\sqrt2}{2}), P_2=(\frac{-\sqrt2}{2},\frac{-\sqrt2}{2}), P_3=(\frac{\sqrt2}{2},\frac{-\sqrt2}{2}), P_4=(\frac{-\sqrt2}{2},\frac{\sqrt2}{2})$. We then have

		\begin{align*}
			<X_1, X_2> &= \frac14 \big(||X_1+X_2||_b^2 - ||X_1-X_2||_b^2\big) \\
			&= \frac14 \bigg( \int_{\mathbb{S}^1} \sign(\omega_1+\omega_2)^2(\omega_1+\omega_2)^2 dH_{\bm X}(\omega) - \int_{\mathbb{S}^1} \sign(\omega_1-\omega_2)^2 (\omega_1 - \omega_2)^2 dH_{\bm X}(\omega) \bigg) \\
            &= \frac14 \bigg( \int_{\mathbb{S}^1 \setminus P_3 \cup P_4} (\omega_1+\omega_2)^2 dH_{\bm X}(\omega) - \int_{\mathbb{S}^1 \setminus P_1\cup P_2} (\omega_1 - \omega_2)^2 dH_{\bm X}(\omega) \bigg) \\
            &= \frac14 \bigg( \int_{\mathbb{S}^1 \setminus \cup P_i} (\omega_1+\omega_2)^2 -(\omega_1-\omega_2)^2 dH_{\bm X}(\omega) + \int_{P_1\cup P_2} (\omega_1 - \omega_2)^2 dH_{\bm X}(\omega) \\ &\qquad- \int_{P_3\cup P_4} (\omega_1 - \omega_2)^2 dH_{\bm X}(\omega) \bigg) \\
            &= \int_{\mathbb{S}^1 \setminus \cup P_i} \omega_2 \omega_2 dH_{\bm X}(\omega) + \frac14 \bigg(\int_{P_1\cup P_2} (\omega_1 - \omega_2)^2 dH_{\bm X}(\omega) -\int_{\setminus P_3\cup P_4} (\omega_1 - \omega_2)^2 dH_{\bm X}(\omega) \bigg).
		\end{align*}
        Now notice the two terms in parentheses above can be written as
        \[
        \int_{P_1\cup P_2} (\omega_1 - \omega_2)^2 dH_{\bm X}(\omega) = 
        \int_{P_1\cup P_2} (2\omega_1)^2 dH_{\bm X}(\omega) = 
        \int_{P_1\cup P_2} 4 \omega_1 \omega_2 dH_{\bm X}(\omega), 
        \]
        and
        \[
        \int_{P_3\cup P_4} (\omega_1 - \omega_2)^2 dH_{\bm X}(\omega) =
        \int_{P_3\cup P_4} (2 \omega_1)^2 dH_{\bm X}(\omega) =
        \int_{P_3\cup P_4} -4 \omega_1 \omega_2 dH_{\bm X}(\omega).
        \]
        Thus, continuing our calculations, we have
        \begin{align*}
            & \quad\int_{\mathbb{S}^1 \setminus \cup P_i} \omega_2 \omega_2 dH_{\bm X}(\omega) + \frac14 \bigg(\int_{P_1\cup P_2} (\omega_1 - \omega_2)^2 dH_{\bm X}(\omega) -\int_{P_3\cup P_4} (\omega_1 - \omega_2)^2 dH_{\bm X}(\omega) \bigg) \\
            &= \int_{\mathbb{S}^1 \setminus \cup P_i} \omega_1 \omega_2 dH_{\bm X}(\omega) + \int_{P_1 \cup P_2} \omega_1 \omega_2 dH_{\bm X}(\omega) + \int_{P_3 \cup P_4} \omega_1 \omega_2 dH_{\bm X}(\omega) \\
			&= \int_{\mathbb{S}^1} \omega_1 \omega_2 dH_{\bm X}(\omega).       
        \end{align*}
		Thus, the inner product is given by the tail pairwise dependence measure.		
	\end{proof}

    As mentioned above, theorem \ref{thm:4} and corollary \ref{cor:1} above make no assumption about the dimension of the underlying subspace on which the inner product exists. As long as any two elements in such a subspace are jointly regularly varying, these results apply. In the next section we discuss some such spaces with a mind toward statistical modeling. Our final corollary notes the assumption $\alpha = 2$ is necessary to obtain an inner product structure.

	\begin{corollary}
		If $\alpha \neq 2$, the norm $||\cdot||_b$ does not satisfy the parallelogram law. 
	\end{corollary}
	\begin{proof}
		Inspection of equation \eqref{eq:4.2.1} in theorem 1 above shows that 
		\begin{align*}
			||X_1+X_2||_b^2 + ||X_1+X_2||_b^2
		\end{align*}
		will not simplify to $2||X_1||_b^2 + 2||X_2||_b^2$ in general. Thus, the norm does not satisfy the parallelogram law in general. 
	\end{proof}
	
	If $\alpha \neq 2$, the norm $||\cdot||_b$ does not come from an inner product, even on an appropriate subspace of $\mathbb{W}_b$. The full space $\mathbb{W}_b$ cannot be an inner product space for any $\alpha \neq 2$; for, if it were, every subspace would also be an inner product space. Moreover, even nicely behaved subspaces like those mentioned above that enjoy the property that any two elements are jointly regularly varying, are not inner product spaces. However, for $\alpha=2$, the situation is still not fully understood. While theorem \ref{thm:4} shows that certain subspaces of $\mathbb{W}_b$ are indeed inner product spaces, extending this to the full space will require different techniques. 

	As a final remark, we mention a connection between the tail pairwise dependence measure as an inner product and the extremal covariance of section \ref{sec:2.2}. The form of the extremal covariance in \ref{eq:2.2.3} is precisely the polarization identity utilized in corollary \ref{cor:1} as the scale functional $S$ is just the norm $||\cdot||_b$ for $\alpha>1$. Indeed, assuming $\alpha=2$ and multivariate regular variation, the tail pairwise dependence measure is exactly the extremal covariance with the constants $a$ and $b$ both set to $1$. In this case, an extremal correlation can be derived using the Cauchy-Schwarz inequality, and would be preferable to the quantity given in \ref{eq:2.2.4}.
    
	\subsection{Subspaces and Modeling Frameworks} \label{sec:4.1}

	Most if not all modeling scenarios involve finitely many observations from finitely many random variables. With this in mind, we begin with finite dimensional subspaces of $\mathbb{W}_b$ that may provide a useful modeling framework for extremes. 
	
	Suppose we have $p$ regularly random variables $\{X_1, ..., X_p\} \subset \mathbb{V}_b$ each with nonzero scale, thought of as observed variables. We may easily form the $p$ dimensional subspace consisting of all linear combinations of them: $S_p := span\{X_1, ..., X_p\} \subset \mathbb{V}_b$. Passing to equivalence classes, we consider $S_p$ a subspace of $\mathbb{W}_b$. The space $S_p$ has a canonical basis given by $B_{S_p} = \{X_1, ..., X_p\}$ inherits the norm $||\cdot||_b$ from $\mathbb{V}_b$, and as a finite dimensional space it is complete. 
    If the dimension of $span\{X_1, ..., X_p\}$ is $p' < p$, (i.e., some of $(X_1, \ldots, X_p)$ are equivalent to linear combinations of the other elements, then identify linearly independent $X_1, \ldots, X_{p'}$ and proceed.

	We then have a finite dimensional Banach space $(S_p, ||\cdot||_b)$, which for simplicity we will assume is $p$ dimensional with basis $\{X_1,...,X_p\}$. If these basis elements arise as observed random variables, then in the context of extremes it seems reasonable to assume that they are jointly regularly varying. Consequently, any element of $S_p$, is jointly regularly varying with any other element (see \cite{BDM02} proposition A.1, or \cite{KulikSoulier2020} proposition 2.1.12). We thus have a Banach space of tail equivalent regularly varying random variables, whose norm has a meaningful connection to the random variables used for modeling, and where any collection of elements thereof are multivariate regularly varying. If the elements of $S_p$ have tail index $\alpha=2$, after marginal transformations if necessary, the results from section \ref{sec:4.1} apply, and $S_p$ carries an inner product that measures extremal dependence. Such a space provides a convenient modeling framework for extremes.
	
	While applied modeling problems will typically be limited to finitely many observable variables, in which case the space $S_p$ will suffice as a modeling framework, larger spaces may provide a more natural setting for specific problems. We sketch one such space with applications in time series. 
	
	Let $\{Z_i\}_{i \geq 1}$ be an i.i.d. sequence of regularly varying random variables with normalizing function $b$ and tail index $\alpha =2$. Define the discrete space
	\[
		\mathbb{D}_b = \bigg\{X: X = \sum_{i=1}^{\infty} \psi_i Z_i, \text{with} \sum |\psi_i| < \infty \bigg\}.
	\]
	One can easily show that $\mathbb{D}_b$ is well defined and is a subspace of $\mathbb{W}_b$. Furthermore, any two elements of $\mathbb{D}_b$ are jointly regularly varying, and so the results of section \ref{sec:4.1} apply (see \cite{HS08}). In particular, the tail pairwise dependence measure acts as an inner product for $\mathbb{D}_b$.	The space $\mathbb{D}_b$ provides a framework for linear processes with heavy tails, and can thus serve as a modeling framework for analysis of heavy tailed time series.

	\section{Discussion} \label{sec:5}

	We proposed a new vector space approach to EVA. The foundation of this new framework is $\mathbb{W}_b$, a normed vector space of (equivalence classes) of tail equivalent regularly varying random variables. We have used $\mathbb{W}_b$ and its defining properties to define new dependence measures that meaningfully track dependence in the extreme tail, while at the same time do not require multivariate regular variation. In addition, we show one such measure is precisely the tail pairwise dependence measure and acts as an inner product for subspaces of $\mathbb{W}_b$ defined using multivariate regular variation.
	
	Using this vector space approach and the tail pairwise dependence measure, we outline an applied modeling framework for EVA in section \ref{sec:4} that mirrors covariance modeling in non-extreme statistics. All the tools from Hilbert space theory are now available for modeling extremes. In particular we remark that the projection theorem can be used for prediction at extreme levels of observed variables. Furthermore, the projection theorem can  be used to define a partial tail dependence measure, in a way comparable to the partial correlation coefficient. This work provides a general setting for these methods, which before had relied on specialized constructions. 
	
	Our work also helps to further explain some existing choices in previous literature. The tail pairwise dependence measure (\ref{eq:1.3.2}) gives a measure of tail dependence for any tail index.
    Several previous works (\cite{CT19}, \cite{mhatre2024}, \cite{lee2021pred}) have restricted attention to the case $\alpha = 2$ because of the tail pairwise dependence matrix's `nice properties'.
    In light of the results in section \ref{sec:4.2}, these nice properties can be understood to come from the nature of the tail pairwise dependence measure as an inner product. Furthermore, the restriction to $\alpha=2$ can now be fully justified: as the tail pairwise dependence measure is not an inner product for any other value of $\alpha$, the restriction to $\alpha=2$ is both convenient and necessary in order for the tail pairwise dependence measure to enjoy desirable properties.
	
	The vector space approach outlined above also raises new research questions. While not strictly necessary for applied modeling, they are of interest in their own right. The first question is whether the space $(\mathbb{W}_b, ||\cdot||_b)$ is complete, that is, whether any of them are Banach spaces. This question is not easily answered. The notion of convergence that $||\cdot||_b$ defines is not easily connected to traditional notions of convergence found in probability theory. This limits the tools available for answering the question of completeness. 
	
	A second question of note is whether any $\mathbb{W}_b$ for a given normalizing function is an inner product space. The corollaries in section \ref{sec:4} provide a partial answer to this question. If $\alpha \neq 2$, the space $(\mathbb{W}_b, ||\cdot||_b)$ cannot be an inner product space. However, if $\alpha=2$ this remains an open question. If this question could be answered in the affirmative, the inner product could be seen as a generalization of the tail pairwise dependence measure beyond the framework of multivariate regular variation. If such an extension exists, it may be of use as a tool for extreme value analysis. It would be one of the few methods that does not depend on the widely used framework of multivariate regular variation.

	\bibliography{VS}

	\appendix
	\section{Omitted Proofs and Details}

	\subsection{Asymptotic Dependence Measures} \label{app:A1}

    This appendix contains proofs of several properties of the extremal covariance and correlation outlined in section \ref{sec:2.2}. First, suppose that $Y=cX$ for some $c\neq 0$. We show that $ER(X,Y)\in \{-1,1\}$. Recalling the definition of $S^2(\cdot)$, we have
    \[
    S^2(aX+bY) = S^2\bigg(\frac{X}{\kappa_X^{\frac{1}{\alpha}}}+\frac{cX}{|c|\kappa_X^{\frac{1}{\alpha}}}\bigg) = \kappa_X^{\frac{2}{\alpha}}\bigg(\kappa^{-\frac{1}{\alpha}}+\sign(c)\kappa^{-\frac{1}{\alpha}} \bigg)^2=
    \begin{cases}
        4 & \sign(c)=1 \\
        0 & \sign(c)=-1
    \end{cases}.
    \]
    A similar computation will yield
    \[
    S^2(aX-bY) = S^2\bigg(\frac{X}{\kappa_X^{\frac{1}{\alpha}}}-\frac{cX}{|c|\kappa_X^{\frac{1}{\alpha}}}\bigg) = \kappa_X^{\frac{2}{\alpha}}\bigg(\kappa^{-\frac{1}{\alpha}}-\sign(c)\kappa^{-\frac{1}{\alpha}} \bigg)^2=
    \begin{cases}
        4 & \sign(c)=-1 \\
        0 & \sign(c)=1 
    \end{cases}.
    \]
    Combining these two in the definition of $ER(X,Y)$ yields
    \[
    ER(X,cX) = 
    \begin{cases}
        1 & \sign(c)=1 \\
        -1 & \sign(c)=-1 \\
    \end{cases}
    \]
    as desired.
    
    We now show the extremal covariance is equal to zero under asymptotic independence. First, note that that $EC(X,Y) =0$ if and only if $S^\alpha(X+Y) = S^\alpha (X-Y)$, or equivalently if and only if $\kappa_{X+Y} = \kappa_{X-Y}$. Now suppose the tails of $X,Y$ satisfy the asymptotic independence conditions:
    \begin{align*}
    b(s)P(X>s,Y>s)\rightarrow 0 \\
    b(s)P(X>s,Y<-s)\rightarrow 0 \\
    b(s)P(X<-s,Y>s)\rightarrow 0 \\
    b(s)P(X<-s,Y<-s)\rightarrow 0 
    \end{align*}
    Under these conditions, it is not difficult to show the following:
    \begin{align*}
    \kappa_{X+Y}^+ &= \kappa_X^+ + \kappa_Y^+ \\
    \kappa_{X+Y}^{-} &= \kappa_X^- + \kappa_Y^- \\
    \kappa_{X-Y}^+ &= \kappa_X^+ + \kappa_Y^- \\
    \kappa_{X-Y}^{-} &= \kappa_X^- + \kappa_Y^+     
    \end{align*}    
    See \cite{JM06} for a discussion of asymptotic independence conditions and the asymptotic behavior of sums. 

    Thus under asymptotic independence
    \begin{align*}
    \kappa_{X+Y} &= \kappa_{X+Y}^++\kappa_{X+Y}^{-} \\
    &= \kappa_X^+ + \kappa_Y^+ + \kappa_X^- + \kappa_Y^- \\
    &= \kappa_{X-Y}^{+} + \kappa_{X-Y}^- \\
    &= \kappa_{X-Y}
    \end{align*}
    Hence, $EC(X,Y)=0$ as desired. In light of the discussion above, it is also not difficult to see that $EC(X,Y)=0$ does not imply asymptotic independence. The $\alpha$-scales can balance in such a way that $EC(X,Y)=0$, but the relations between the left and right $\alpha$-scales under asymptotic independence do not hold.
	
\end{document}